\documentclass[reqno,11pt]{amsart}
\usepackage{amssymb,enumerate,bbm}
\numberwithin{equation}{section}

\newtheorem{theorem}[equation]{Theorem}
\newtheorem{proposition}[equation]{Proposition}

\newtheorem{lemma}[equation]{Lemma}
\newtheorem{remark}[equation]{Remark}

\newcommand{\N}{\mathbbmss{N}}
\newcommand{\Q}{\mathbbmss{Q}}
\newcommand{\Z}{\mathbbmss{Z}}
\newcommand{\g}{\widehat{g}}
\newcommand{\f}{\widehat{f}}
\renewcommand{\mid}{\,|\,}
\newcommand{\ndiv}{\not|\,}

\title[{Factorization of quadratic polynomials in $\Z[[x]]$}]%
{Factorization of quadratic polynomials in the ring of
formal power series over $\Z$}
\author{Daniel Birmajer}
\address{Department of Mathematics\\ Nazareth College\\
4245 East Avenue\\ Rochester, NY 14618}
\email{abirmaj6@naz.edu}
\author{Juan B. Gil}
\address{Penn State Altoona\\ 3000 Ivyside Park\\ Altoona, PA 16601.}
\email{jgil@psu.edu}
\author{Michael D. Weiner}
\address{Penn State Altoona\\ 3000 Ivyside Park\\ Altoona, PA 16601.}
\email{mdw8@psu.edu}

\begin{document}

\begin{abstract}
We establish necessary and sufficient conditions for a quadratic 
polynomial to be irreducible in the ring $\Z[[x]]$ of formal power series 
over the integers. In particular, for polynomials of the form 
$p^n+p^m\beta x+\alpha x^2$ with $n,m\ge 1$ and $p$ prime, we show that 
reducibility in $\Z[[x]]$ is equivalent to reducibility in $\Z_p[x]$, 
the ring of polynomials over the $p$-adic integers. 
\end{abstract}

\maketitle
\newcounter{prfeqn}

%%%%%%%%%%%%%%%%%%%%%%%%%%%%%%%%%%%%%%%%%%%%%%%%%%%%%%%%%%%%%%%%%%%%%
\section{Introduction}

If $K$ is a field, $\text{char}\,K\not=2$, the question of whether or not
a quadratic polynomial is reducible in the polynomial ring $K[x]$ is well 
understood: A polynomial $f(x)=c+bx+ax^2$, with $a\ne 0$, can be written 
as a product of two linear factors in $K[x]$ if and only if its 
discriminant $b^2-4ac$ is a square in $K$. Moreover, by Gauss' Lemma,
if $D$ is a unique factorization domain with field of fractions $K$, then 
a primitive quadratic polynomial in $D[x]$ is reducible if and only it is 
reducible in $K[x]$.
\par
If we consider the polynomials in $\Z[x]$ as elements of $\Z[[x]]$, the 
ring of formal power series over $\Z$, the factorization theory has a 
different flavor. A power series over an integral domain $D$ is a unit 
in $D[[x]]$ if and only if its constant term is a unit in $D$, so 
irreducible elements in $\Z[x]$, such as $1+x$, are invertible as power
series. On the other hand, any power series whose constant term is not a 
unit or a prime power, is reducible in $\Z[[x]]$, hence we can produce 
many examples of polynomials that are reducible as power series, yet 
irreducible in $\Z[x]$.
\par
Similarly, when considering polynomials with integer coefficients as 
elements of $\Z[[x]]$ and as polynomials over $\Z_p$, the ring of $p$-adic 
integers, we also observe different behaviors in their arithmetic
properties. For instance, the polynomial $p^2+x+x^2$, which is irreducible 
as a power series, is reducible in $\Z_p[x]$ for any prime $p$. On the 
other hand, $6+2x+x^2$ is reducible in $\Z[[x]]$ and in $\Z_3[x]$, 
but it is irreducible in $\Z_2[x]$ and $\Z_5[x]$. 
\par 
In this paper, we provide a complete picture of the factorization theory 
for quadratic polynomials in $\Z[[x]]$. In Section~\ref{s:factorization}
we discuss the necessary background, treat some basic cases, and develop 
some preliminary results. In Section~\ref{sec:OddPrimePower} we study 
polynomials of the form $p^n+p^m\beta x +\alpha x^2$ (the only ones not 
discussed in Section~\ref{s:factorization}) and show a revealing connection
between $\Z[[x]]$ and $\Z_p[x]$, cf. Theorem~\ref{t:main}. 
In Section~\ref{sec:FurtherCriteria} we extend our results to power series
and give some reducibility criteria that rely on the knowledge of their 
quadratic part.
\par
A standard reference for an introduction to divisibility over integral 
domains is \cite{DuFo}. For an extensive treatment of the arithmetic on 
the ring of formal power series over an integral domain the reader is 
referred to \cite{Ei95} and \cite{Ka70}. All the necessary 
material about the ring $\Z_p$ of $p$-adic numbers, can be found for 
instance in \cite{BoSh,Ei95,Serre}.

%%%%%%%%%%%%%%%%%%%%%%%%%%%%%%%%%%%%%%%%%%%%%%%%%%%%%%%%%%%%%%%%%%%%%
\section{Factorization in the ring of power series}\label{s:factorization}

In order to place our main result in the appropriate context, and for the
reader's convenience, we review some elementary facts about the
factorization theory in $\Z[[x]]$. First, recall that $\Z[[x]]$ is a
unique factorization domain. Moreover, if $f(x)$ is a formal power series
in $\Z[[x]]$ and $f_0\in \Z$ is its constant term, then: 
\begin{enumerate}[(a)]
\item $f(x)$ is invertible if and only if $f_0=\pm 1$.
\item If $f_0$ is prime then $f(x)$ is irreducible.
\item If $f_0$ is not a unit or a prime power then $f(x)$ is reducible.
\item If $f(x)=f_0$ is a constant then it is irreducible if and only if
$f_0$ is prime.
\item If $f(x)=p^m+f_1 x$, with $p$ prime and $m\ge 1$, 
then $f(x)$ is irreducible if and only if $\gcd(p, f_1)=1$.
\end{enumerate}

For an accessible and  more detailed treatment of the divisibility theory
in $\Z[[x]]$ the reader is referred to \cite{BiGi}.

The above criteria are definitive for deciding irreducibility 
in $\Z[[x]]$ for constant and linear polynomials. The next natural step is 
to examine quadratic polynomials 
\begin{equation}\label{GenericQuadratic}
  f(x)=f_0+f_1 x + f_2 x^2\quad \text{with }\; f_2\ne 0,\; f_i\in \Z. 
\end{equation}
Unless $f_0$ is a prime power, we know that $f(x)$ is either a unit or it
is reducible in $\Z[[x]]$. On the other hand, if $f_0=p^n$, $n> 1$, $p$ 
prime, and if $f(x)=a(x)b(x)$ is a proper factorization, then we must have 
$a_0=p^s, b_0=p^t$ with $s,t\ge 1$, $s+t=n$.  This implies 
$f_1=p^s b_1+p^t a_1$, so we conclude that $f(x)$ is irreducible
unless $p\mid f_1$. Therefore, for a quadratic polynomial we have:
\begin{enumerate}[(a)]
\setcounter{enumi}{5}
\item If $f_0=p^n$, $n>1$, and $p \ndiv f_1$, then $f(x)$ is irreducible 
in $\Z[[x]]$.
\item If $p$ divides $f_0$, $f_1$, and $f_2$, then $f(x)$ is either 
reducible or associate to $p$.  
\end{enumerate}

\medskip
At this point, it only remains to understand polynomials of the form
\[ f(x)=p^n+ p^m\beta x +\alpha x^2, \]
with $n,m\ge 1$, $\gcd(p,\alpha)=1$, and $\gcd(p,\beta)=1$ or $\beta=0$.
We will analyze these polynomials in the next section and will produce 
explicit factorizations in $\Z[[x]]$, when appropriate. 
Our results will provide the following additional irreducibility criterion
for the polynomial \eqref{GenericQuadratic}:
\begin{enumerate}[(a)]
\setcounter{enumi}{7}
\item If $f_0=p^n$, $n>1$, $p\mid f_1$, and $\gcd(p,f_2)=1$, then $f(x)$ is 
reducible in $\Z[[x]]$ if and only if it is reducible in $\Z_p[x]$.
\end{enumerate}

Note that items (a)--(c) and (f)--(h)  give complete irreducibility criteria 
for quadratic polynomials in $\Z[[x]]$. 
%%%%%%%%%%%%%%%%%%%%%%%%%%%%%%%%%%%%%%%%%%%%%%%%%%%%%%%%%%%%%%%%%%%%%
\section{Polynomials of the form $p^n+ p^m\beta x +\alpha x^2$}
\label{sec:OddPrimePower}

Let $p$ be an odd prime, let $\alpha,\beta\in\Z$ be such that
$\gcd(p,\alpha)=1$ and $\gcd(p,\beta)=1$.

\begin{proposition}\label{p-easycase}
Let $f(x)=p^n+ p^m\beta x +\alpha x^2$ with $n,m\ge 1$.
\begin{enumerate}[$(i)$]
\item 
If $2m<n$, then $f(x)$ is reducible in both $\Z_p[x]$ and $\Z[[x]]$.
\item 
If $2m>n$ and $n$ is odd, then $f(x)$ is irreducible in both $\Z_p[x]$ 
and $\Z[[x]]$.
\end{enumerate}
\end{proposition}
\begin{proof}
$(i)$ Observe first that the discriminant of $f(x)$ is
\begin{equation*}
 p^{2m}\beta^2-4\alpha p^n=p^{2m}(\beta^2-4\alpha p^{n-2m}),
\end{equation*}
a nonzero square in $\Z_p$, and so $f(x)$ is reducible in $\Z_p[x]$.
To show that $f(x)$ is reducible as a power series, we will find
sequences $\{a_k\}$ and $\{b_k\}$ such that
\begin{equation*}
f(x)=(p^m+a_1x+a_2x^2+\cdots)(p^{n-m}+b_1x+b_2x^2+\cdots).
\end{equation*}
For $k\ge 1$ let $t_k=b_k+p^{n-2m}a_k$. For the above factorization to
hold, we need 
\[ p^m\beta=p^{n-m}a_1+p^m b_1, \;\text{ so we have }
   t_1=\beta. \]
Let $g(x)=p^{n-2m} x^2 -\beta x+\alpha$. Since $\gcd(p,\beta)=1$, this 
polynomial has a root in $\Z/p\Z$ while $g'(x)=2p^{n-2m} x-\beta$ has 
none. By Hensel's Lemma $g(x)$ has a root in $\Z_p[x]$ and so, 
in particular, there are integers $a_1$ and $t_2$ such that
\[ p^{n-2m} a_1^2 -\beta a_1+\alpha = p^m t_2. \]

Suppose that we have defined $a_k$, $t_{k+1}$ for $k=1,\dots,N-1$,
$N\ge 2$, and let
\begin{equation*}
v_N=a_1t_{N}+\sum_{k=2}^{N-1}a_k(t_{N+1-k}-p^{n-2m}a_{N+1-k}).
\end{equation*}
We want to define $a_N$ and $t_{N+1}$ in such a way that
$\sum_{k=0}^{N+1}a_kb_{N+1-k}=0$ for $N\ge 2$. In other words, we need
\begin{align*}
0 &= \sum_{k=0}^{N+1} a_k b_{N+1-k} \\
  &=a_0b_{N+1}+a_{N+1}b_0+a_{N}b_1+a_1b_{N}+\sum_{k=2}^{N-1}a_kb_{N+1-k}\\
  &=p^m t_{N+1}+(\beta-p^{n-2m}a_1)a_{N}+a_1(t_{N}-p^{n-2m}a_{N})
    +\sum_{k=2}^{N-1}a_kb_{N+1-k}\\
  &=p^m t_{N+1}+(\beta-2p^{n-2m}a_1)a_{N}+v_N.
\end{align*}
At last, since $\gcd(p,\beta)=1$, this equation
can be solved for $t_{N+1}, a_{N}\in \Z$. This shows that $f(x)$ is 
reducible in $\Z[[x]]$.

\medskip
$(ii)$ In this case, the discriminant of $f(x)$,  
\begin{equation*}
 p^{2m}\beta^2-4\alpha p^n=p^n(p^{2m-n}\beta^2-4\alpha), 
\end{equation*}
is not a square in $\Z_p$.  Thus $f(x)$ is irreducible as a polynomial
over $\Z_p$. To show that $f(x)$ is irreducible as a 
power series, assume
\begin{equation*}
f(x)=(p^s+a_1x+a_2x^2+\cdots)(p^t+b_1x+b_2x^2+\cdots)
\end{equation*}
with $t>s\ge 1$, $s+t=n$. Note that $t\not=s$ because $n$ is odd.  
Then we must have
\begin{align*}
p^m\beta &=p^ta_1+p^sb_1, \\
\alpha &=p^ta_2+a_1b_1+p^sb_2. 
\end{align*}
Since $p$ and $\alpha$ are coprime, it follows that 
$\gcd(p,a_1)=1=\gcd(p,b_1)$. Therefore, it must be $s=m$, and so 
$2m=2s< s+t=n$.  
\end{proof}

It remains to analyze the cases when $n$ is even, say $n=2\nu$,
and $m\ge \nu\ge 1$. 

\begin{proposition}\label{p-hardcase}
Let $m\ge \nu$.
The polynomial $f(x)=p^{2\nu}+p^m\beta x +\alpha x^2$ is reducible in
$\Z[[x]]$ if and only if $\f(x)= p^2 + p^{m-\nu+1}\beta x +\alpha x^2$ 
is reducible in $\Z_p[x]$.
\end{proposition}

This follows from the following three lemmas.

\begin{lemma}\label{lemmaA}
If $f(x)$ is reducible in $\Z[[x]]$, then $\f(x)$ is reducible in
$\Z[[x]]$.
\end{lemma}

\begin{lemma}\label{lemmaB}
Let $\ell\ge 1$.  If the polynomial $p^2 + p^{\ell}\beta x +\alpha x^2$ 
is reducible in $\Z[[x]]$, then it is reducible in $\Z_p[x]$.
\end{lemma}

\begin{lemma}\label{lemmaC}
If $\f(x)$ is reducible in $\Z_p[x]$, then $f(x)$ is reducible in
$\Z[[x]]$.
\end{lemma}

\begin{proof}[\bf Proof of Lemma~\ref{lemmaA}]
We first observe that if $f(x)=a(x)b(x)$ is a proper factorization  
in $\Z[[x]]$, then $a_0=b_0=p^\nu$.
To see this, assume that $a_0=p^s$, $b_0=p^t$ with $s, t\ge 1$, $s+t=2\nu$.
Then we have that
\begin{equation*}
\alpha=p^s b_2+a_1 b_1 + p^t a_2.
\end{equation*}
Since $\gcd(p,\alpha)=1$, we conclude that $\gcd(p,a_1)= \gcd(p,b_1)=1$.
We also have 
\begin{equation*}
p^m\beta=p^s b_1+p^t a_1.
\end{equation*}
If $s<t$, then we would have $s<\nu\le m$, implying from the above
equation that  $p\mid b_1$, a contradiction.
Similarly, we can rule out the case $t<s$, hence $s=t=\nu$.

We now write $f(x)=a(x)b(x)$ with $a_0=b_0=p^{\nu}$. Since
\begin{equation*}
  p^{2\nu-2}\f(x)=f(p^{\nu-1}x)=a(p^{\nu-1}x)b(p^{\nu-1}x),
\end{equation*}
it follows that
\begin{equation*}
 \f(x)=\Big(\frac{a(p^{\nu-1}x)}{p^{\nu-1}}\Big)
 \Big(\frac{b(p^{\nu-1}x)}{p^{\nu-1}}\Big)
\end{equation*}
is a proper factorization of $\f(x)$ in $\Z[[x]]$.
\end{proof}

\begin{proof}[\bf Proof of Lemma~\ref{lemmaB}]
To prove that $g(x)=p^2 + p^{\ell}\beta x +\alpha x^2$ is reducible in
$\Z_p[x]$, we must show that its discriminant 
$p^{2\ell}\beta^2-4\alpha p^2$ is a square in $\Z_p$. 

If the discriminant is zero, we are done. Otherwise,
write $p^{2\ell-2}\beta^2-4\alpha=p^t u$ with $\gcd(p,u)=1$.
Suppose that $g(x)$ is reducible in $\Z[[x]]$.  
Without loss of generality we can assume that $g(x)$ admits a
factorization of the form
\begin{equation*}
p^2 + p^{\ell}\beta x +\alpha x^2=a(x)b(x)\quad\text{with }\;
a_0=b_0=p,\;\; a_2=a_3=\cdots =a_{t+2}=0.
\end{equation*}
With the notation $s_j=a_j+b_j$ for $j\ge 1$, 
we must have 
\begin{align*}
p^{\ell}\beta &=ps_1, \\
\alpha &= ps_2+a_1s_1-a_1^2.
\end{align*}
Then $s_1=p^{\ell-1}\beta$ and $a_1$ is a root of
$y^2-s_1y+\alpha\equiv 0\pmod p$. Note that $p\ndiv a_1$.

For $n=3$ we have
\begin{equation}\label{eq:base}
0=ps_3+a_1s_2.
\end{equation}
Then $p\mid s_2$ and $a_1^2-a_1s_1+\alpha\equiv 0 \pmod{p^2}$. 
For $n=4$ we have
\begin{equation*}
0=ps_4+a_1s_3.
\end{equation*}
Then $p\mid s_3$, which by \eqref{eq:base} implies that 
$p^2\mid s_2$, and so $a_1^2-a_1s_1+\alpha\equiv 0 \pmod{p^3}$. 
Working inductively, the equation
\begin{align*}
0= ps_{t+3}+a_1s_{t+2}
\end{align*}  
implies that $p^{t+1}\mid s_2$, and so $a_1^2-a_1s_1+\alpha= 
p^{t+2}v$ for some $v$.

Now, since
\[ (2a_1-s_1)^2 = (p^{2\ell-2}\beta^2-4\alpha) + 4p^{t+2}v 
   = p^t u +4p^{t+2}v=p^t(u+4p^2v), \]
and since $\gcd(p,u)=1$, we have that $t$ is even 
and that $u$ is a square mod $p$. Hence $p^{2(\ell-1)}\beta^2-4\alpha$ 
is a square in $\Z_p$, and so is $p^2(p^{2(\ell-1)}\beta^2-4\alpha)$.
Therefore, $g(x)$ is reducible in $\Z_p[x]$.
\end{proof}

\begin{proof}[\bf Proof of Lemma~\ref{lemmaC}]
We will consider the cases $m=\nu$ and $m>\nu$ separately. 
In both cases we will prove the reducibility of $f(x)$ in $\Z[[x]]$
by providing an explicit factorization algorithm. More precisely, 
we will give inductive algorithms (depending on $m$ and $\nu$) to find 
sequences $\{a_k\}$ and $\{b_k\}$ in $\Z$ such that 
\begin{equation}\label{fFactor}
   f(x)=\Big(\sum_{k=0}^\infty a_k x^k\Big) 
   \Big(\sum_{k=0}^\infty b_k x^k\Big). 
\end{equation}
For $k\ge 1$ we let $s_k=a_k+b_k$. 

\medskip\noindent
{\sc Case 1:} Let $m>\nu$. Since $\f(x)=p^2+p^{m-\nu+1}\beta x 
+\alpha x^2$ is reducible in $\Z_p[x]$, the polynomial 
$\g(x)=x^2-p^{m-\nu}\beta x+\alpha$ is reducible in $\Z_p[x]$,
too. Observe that the discriminant of $\f(x)$ is $p^2$ times the
discriminant of $\g(x)$.

Let 
\[ a_0=p^\nu=b_0 \quad\text{and}\quad s_1=p^{m-\nu}\beta. \]
Since $\g(x)$ is reducible in $\Z_p[x]$, it has a root in
$\Z/p^{\nu}\Z$. Let $a_1,s_2\in\Z$ be such that
\[ a_1^2 - p^{m-\nu}\beta a_1 +\alpha = p^\nu s_2. \]
Now, $m>\nu$ and $\gcd(p,\alpha)=1$ imply
$\gcd(p,a_1)=1$ and $\gcd(p^\nu,p^{m-\nu}\beta-2a_1)=1$.
We let $a_2$ and $s_3$ be integer numbers such that
\[ 0=p^\nu s_3 + (p^{m-\nu}\beta-2a_1)a_2 + a_1s_2. \]
Suppose we have defined $a_{k}$ and $s_{k+1}$ for $k=1,\dots,N-1$, 
$N\ge 3$, and let
\[ v_N=a_1s_N +\sum_{k=2}^{N-1}a_k(s_{N+1-k}-a_{N+1-k}). \]
We know that $\gcd(p^\nu,p^{m-\nu}\beta-2a_1)=1$, so the equation
\[ 0=p^\nu s_{N+1}+(p^{m-\nu}\beta-2a_1)a_N + v_N \]
can be solved for $a_N, s_{N+1}\in\Z$. 
For $k=1,\dots,N$ we now have $a_k$ and $b_k$, and
it can be easily checked that the sequences $\{a_k\}$ 
and $\{b_k\}$ give \eqref{fFactor}.

\medskip\noindent
{\sc Case 2:} If $m=\nu$,
then 
\begin{equation*}
\f(x)=p^2 + p\beta x +\alpha x^2 \quad\text{and}\quad
 f(x)=p^{2\nu} + p^{\nu}\beta x +\alpha x^2.
\end{equation*}
Since $\f(x)$ is reducible in $\Z_p[x]$, so is $\g(x)=x^2-\beta x+\alpha$.
If $\beta^2-4\alpha=0$, then $\beta$ is even and
$f(x)=\big(\frac{\beta}{2}x+p^{\nu}\big)^2$, which is a proper 
factorization in $\Z[[x]]$. If $\beta^2-4\alpha\not=0$,     
there are numbers $\ell\in\N_0$ and $q\in\Z$ such that 
\begin{equation*}
 \beta^2-4\alpha=p^{2\ell}q 
 \;\text{ with } \gcd(p,q)=1. 
\end{equation*}
Moreover, $\g(x)$ has a root in $\Z/p^{n}\Z$ for every $n\in\N$.
In particular, for $n=3\max(\ell,\nu)$,
there are integers $a$ and $r$ such that  
\begin{equation}\label{RootofP1}
 a^2 -\beta a +\alpha = p^{\mu} r
 \;\text{ with } \gcd(p,r)=1,
\end{equation}
for some $\mu\ge 3\max(\ell,\nu)$. 
Since $(\beta-2a)^2-(\beta^2-4\alpha)=4\g(a)$, we get
\begin{equation*}
 (\beta-2a)^2=4p^\mu r + p^{2\ell}q=p^{2\ell}(4p^{\mu-2\ell}r + q),
\end{equation*}
hence we can write 
\begin{equation}\label{RootofP2}
 \beta-2a=p^{\ell}t \;\text{ with } \gcd(p,t)=1.
\end{equation}

Again, our goal is to construct sequences $\{a_k\}$ and $\{b_k\}$ such
that \eqref{fFactor} holds.  This will be done with slightly different
algorithms for $\nu>\ell$ and $\nu\le\ell$. In both cases we let
\begin{gather*} 
a_0=p^\nu=b_0, \quad s_1=\beta, \\
a_1=a, \quad s_2= p^{\mu-\nu}r,
\end{gather*}
where $a$ and $r$ are the integers from \eqref{RootofP1}.
With these choices, the first three terms in the expansion of
\eqref{fFactor} coincide with $f(x)$. 

Assume $\nu>\ell$. Let 
\[ \tilde a_1=0, \quad u_1=s_2=p^{\mu-\nu}r, \;\;\text{ and }\;\;
   u_2=-p^{\mu-2\nu}ra_1.
\]
Let $t$ be as in \eqref{RootofP2}.
For $k\ge 2$ we will define $\tilde a_k$ and 
$u_{k+1}$ such that the sequences defined by 
\begin{equation}\label{FactorSeq1}
   a_k = p^{\nu-\ell}\tilde a_k \quad\text{and}\quad
   b_k = u_{k-1}-t\tilde a_{k-1}-a_k 
\end{equation}  
give the factorization \eqref{fFactor}. Note that 
$s_{k+1}=a_{k+1}+b_{k+1}=u_{k}-t\tilde a_{k}$.  

Let $\tilde a_2= p^{\mu-2\nu}$ and
$u_3= -p^{\mu-3\nu}\big[p^{\nu-\ell}(s_2-a_2)-ta_1\big]$. Thus
\[ p^\nu u_3 + \big[p^{\nu-\ell}(s_2-a_2)-ta_1\big]\tilde a_2=0. \]
Suppose we have defined $\tilde a_k$ and $u_{k+1}$ for $k=1,\dots,N-1$,
$N\geq 3$, and let 
\[ v_N= a_1 u_N + a_2 s_N + \sum_{k=3}^{N-1}a_k(s_{N+2-k}-a_{N+2-k}). \]
Since $\gcd(p,\beta)=1$, the relation \eqref{RootofP2} implies
\[ \gcd(p,a_1)=1 \;\text{ and }\; \gcd(p^{\nu-\ell},
   p^{\nu-\ell}(s_2-2a_2)-ta_1)=1. 
\]
Therefore, there are $\tilde a_N, u_{N+1}\in \Z$ such that    
\[ p^\nu u_{N+1} + \big[p^{\nu-\ell}(s_2-2a_2)-ta_1\big]\tilde a_N 
   + v_N = 0. 
\]
The sequences $\{a_k\}$ and $\{b_k\}$ defined by \eqref{FactorSeq1} 
give \eqref{fFactor} when $\nu>\ell$.

Assume now $\nu\le\ell$. In this case, for $k\ge 2$ we will find 
$\tilde a_k$ and $\tilde s_{k+1}$ such that the sequences defined by
\begin{equation*}
a_k=p^\ell \tilde a_k \quad\text{and}\quad
b_k=p^{3\ell-\nu}\tilde s_k-a_k
\end{equation*}
give a factorization of $f(x)$. Let $r$ and $t$ be as in 
\eqref{RootofP1} and \eqref{RootofP2}, respectively.
Since $\gcd(p^\nu,t)=1$, there are $y,z\in\Z$ such that
\[ p^\nu y+tz + r=0. \]
Let $\tilde a_2=p^{\mu-2\ell-\nu}z a_1$, $\tilde s_2=p^{\mu-3\ell}r$,
and $\tilde s_3= p^{\mu-3\ell}y a_1$. Note that
\[ p^\ell \tilde s_3 + t\tilde a_2 + a_1p^{\ell-\nu}\tilde s_2 =0. \]
Suppose we have defined $\tilde a_k$ and $\tilde s_{k+1}$ for 
$k=1,\dots,N-1$, $N\geq 3$, and let
\[ \tilde v_N= a_1 p^{\ell-\nu}\tilde s_N + \sum_{k=2}^{N-1}
   \tilde a_k(p^{2\ell-\nu}\tilde s_{N+1-k}-\tilde a_{N+1-k}). 
\]
Finally, since $\gcd(p^\ell,t)=1$, the equation
\[ 0=p^\ell \tilde s_{N+1} + t\tilde a_N + \tilde v_N \]
can be solved for $\tilde a_N, \tilde s_{N+1}\in \Z$. This implies
\begin{align*}
0 
 &=p^{3\ell} \tilde s_{N+1} + p^{2\ell}t\tilde a_N + p^{2\ell}\tilde v_N\\ 
 &=p^{\nu} s_{N+1} +p^{\ell}t a_N + a_1s_N +  
   \sum_{k=2}^{N-1} a_k(s_{N+1-k}-a_{N+1-k}) \\
 &=p^{\nu} s_{N+1} +(\beta-2a_1) a_N + a_1s_N +  
   \sum_{k=2}^{N-1} a_k b_{N+1-k} 
 =\sum_{k=0}^{N+1} a_k b_{N+1-k},
\end{align*}
as desired. This completes the proof.
\end{proof}

The main result of this section is the following.

\begin{theorem}\label{t:main}
Let $p$ be an odd prime and let $n,m\ge 1$. Let
$\alpha,\beta\in\Z$ be such that $\gcd(p,\alpha)=1$ and $\gcd(p,\beta)=1$.
The polynomial $f(x)=p^n+ p^m\beta x +\alpha x^2$
is reducible in $\Z[[x]]$ if and only if it is reducible in $\Z_p[x]$. 
\end{theorem}
\begin{proof}
Using the fact that $f(x)$ is reducible in $\Z_p[x]$ iff $\f(x)$ is 
reducible in $\Z_p[x]$, the statement of the theorem follows from 
Proposition~\ref{p-easycase} and Proposition~\ref{p-hardcase}.
\end{proof}

\begin{remark}
The previous theorem is not valid when $m=0$. In fact, if $p\ndiv \beta$, 
any power series of the form $p^n+\beta x+\cdots$ is irreducible in
$\Z[[x]]$. However, any polynomial $p^n+\beta x+\alpha x^2$ with
$\gcd(p,\beta)=1$ is reducible in $\Z_p[x]$.
\end{remark}

We finish this section with the remaining case: $\beta=0$.

\begin{proposition}\label{case:pbeta=0}
Let $p$ be an odd prime and let $n\ge 1$. Let $\alpha\in\Z$ be such that 
$\gcd(p,\alpha)=1$. The polynomial $f(x)=p^n+ \alpha x^2$ is reducible 
in $\Z[[x]]$ if and only if it is reducible in $\Z_p[x]$.
\end{proposition}
\begin{proof}
Recall that $f(x)=p^n+ \alpha x^2$ is reducible in $\Z_p[x]$
if and only if its discriminant $-4\alpha p^n$ is a nonzero square in 
$\Z_p$. This in turn is the case if and only if $n$ is even and 
$-\alpha$ is a square in $\Z/p\Z$. We will show that these conditions
on $n$ and $\alpha$ are equivalent to $f(x)$ being reducible in $\Z[[x]]$.

For $f(x)$ to admit a factorization of the form
\[ 
   p^n+\alpha x^2 = (a_0+a_1x+a_2x^2+\cdots)(b_0+b_1x+b_2x^2+\cdots)
\]
it is necessary to solve the equations
\begin{gather*}
a_0=p^t \text{ and } b_0=p^s \;\text{ with } t+s=n, \\
0=p^t b_1 + p^s a_1, \\
\alpha = p^t b_2 + a_1b_1 + p^s a_2.
\end{gather*}
Since $\gcd(p,\alpha)=1$, these three equations can be solved in $\Z$ 
only when $s=t$, that is, when $n$ is even. Now, if
$n=2\nu$, we must have $a_0=b_0=p^\nu$, $s_1=a_1+b_1=0$, and 
$\alpha=p^\nu(a_2+b_2)-a_1^2$. Thus, if $f(x)$ is reducible
in $\Z[[x]]$, then $-\alpha$ is a square in $\Z/p\Z$. On the other
hand, if $-4\alpha p^{2\nu}$ is a nonzero square in $\Z_p$,  
so is $-\alpha$, i.e., $y^2 +\alpha$ has a root in $\Z_p$. Let $a_1$
and $s_2$ be integers such that
\[ a_1^2 + \alpha = p^{\nu} s_2. \]
Note that $\gcd(p^\nu,2a_1)=1$. 
Therefore, there are integers $a_2$ and $s_3$ such that
\[ 0=p^{\nu}s_3 -2a_1a_2 + a_1s_2. \]
Finally, a factorization of $f(x)$ in $\Z[[x]]$ can be obtained with 
the sequences $\{a_k\}$ and $\{s_{k+1}\}$ defined inductively
for $N\ge 3$ by the equation
\[ 0= p^\nu s_{N+1} - 2a_1a_N + v_N, \]
where $v_N=a_1s_N +\sum_{k=2}^{N-1}a_k(s_{N+1-k}-a_{N+1-k})$.
\end{proof}

\begin{remark}
With the appropriate adjustments in the proofs, all the results in this 
section apply verbatim to the case when $p=2$.  For the interested
reader, we recall that an element $2^nu\in\Q_2^*$ is a square iff $n$
is even and $u\equiv 1\pmod 8$.
\end{remark}

%%%%%%%%%%%%%%%%%%%%%%%%%%%%%%%%%%%%%%%%%%%%%%%%%%%%%%%%%%%%%%%%%%%%%
\section{Further reducibility criteria}
\label{sec:FurtherCriteria}

In this last section we briefly discuss the factorization in
$\Z[[x]]$ of power series whose quadratic part is a polynomial
like the ones studied in the previous sections. More precisely, 
we consider power series of the form
\begin{equation}\label{GenPowerSeries}
 f(x)=p^n + p^m\beta x + \alpha x^2 +\sum_{k=3}^\infty c_k x^k, 
\end{equation}
where $\alpha$ and $\beta$ are integers such that
$\gcd(p,\alpha)=1=\gcd(p,\beta)$. 

For simplicity, we only discuss the case when $p$ is an odd prime.
We will focus on the situations for which the arguments in 
Section~\ref{sec:OddPrimePower} extend with little or no additional 
effort. For instance, if $m\not=\frac{n}{2}$, the reducibility of $f(x)$ 
in $\Z[[x]]$ follows the same pattern as the reducibility of its 
quadratic part.  In fact, we can use the exact same arguments from
Section~\ref{sec:OddPrimePower} to prove the following two propositions.

\begin{proposition}
If $2m<n$, then \eqref{GenPowerSeries} is reducible in $\Z[[x]]$. 
If $2m>n$ and $n$ is odd, then \eqref{GenPowerSeries} is irreducible.
\end{proposition}

\begin{proposition}
If $2m>n$ and $n$ is even, then \eqref{GenPowerSeries} is reducible in 
$\Z[[x]]$ if and only if $-\alpha$ is a quadratic residue \textup{mod} $p$.
\end{proposition}

If $2m=n$, the situation is in general more involved and the reducibility
of $f(x)$ depends on the roots of $x^2-\beta x +\alpha$. 
The following proposition is easy to prove.

\begin{proposition}
If $2m=n$ and the polynomial $x^2-\beta x +\alpha$ has a simple root
in $\Z/p^m\Z$, then \eqref{GenPowerSeries} is reducible in $\Z[[x]]$. 
\end{proposition}

If $x^2-\beta x +\alpha$ has a double root in $\Z/p^m\Z$, it is not
enough to look at the quadratic part of $f(x)$ and its reducibility
depends on the coefficients $c_k$. To illustrate this fact,
consider for example the power series 
\[ f(x)=p^2+p\beta x+\alpha x^2 + c_3x^3 + c_4x^4+ \cdots, \]
with $\alpha,\beta\in\Z$ such that $\beta^2-4\alpha=p^2 q$, where $q$ is 
a quadratic residue mod $p$ with $\gcd(p,q)=1$. 
In order to get a proper factorization
$f(x)=a(x)b(x)$ in $\Z[[x]]$, we must have $a_0=p=b_0$, 
$\beta=s_1$, as well as 
\begin{align*}  
 \alpha &= ps_2+a_1(\beta-a_1), \\ 
 c_3 &= ps_3 + (\beta-2a_1)a_2 + a_1s_2,
\end{align*}
where $s_k=a_k+b_k$. Then $(\beta-2a_1)^2-(\beta^2-4\alpha)=4ps_2$, 
which implies $p\mid s_2$. Therefore, $f(x)$ is irreducible in $\Z[[x]]$
unless $p\mid c_3$.

On the other hand, if $p^2\mid c_k$ for every $k\ge 3$, then with the 
same assumptions on $\alpha$ and $\beta$ as above, we can find 
$a_k,b_k\in\Z$ such that $a(x)=\sum a_kx^k$ and $b(x)=\sum b_kx^k$ give 
a proper factorization $f(x)=a(x)b(x)$. Note that $\beta^2-4\alpha$
is a square in $\Z_p$, so the polynomial $g(x)=x^2-\beta x +\alpha$ 
is reducible in $\Z_p[x]$.  In particular, $g(x)$ has a root in 
$\Z/p^3\Z$, so there are $a,r\in\Z$ such that
\[ a^2 - \beta a +\alpha = p^3 r. \]
Moreover, since $(\beta-2a)^2-(\beta^2-4\alpha)=4p^3r$ and $p^2\mid
(\beta^2-4\alpha)$, we have $p\mid(\beta-2a)$. In fact, there is an
integer $t$ with $\gcd(p,t)=1$ such that 
\[ \beta-2a=pt. \] 
Choose $a_1=a$, $\,\tilde s_2=r$, and write $c_{k+1}= p^2\tilde c_{k+1}$.
Since $\gcd(p,t)=1$, for $k\ge 2$ we can choose $\tilde a_k$ and 
$\tilde s_{k+1}$ inductively as integer solutions of the equation
\[ \tilde c_{k+1}=p\tilde s_{k+1}+t\tilde a_k+ a_1\tilde s_k +
   \sum_{j=2}^{k-1} \tilde a_j(p\tilde s_{k+1-j}-\tilde a_{k+1-j}). \]
If we let $a_k=p\tilde a_k$ and $s_k=p^2\tilde s_k$, then multiplication
by $p^2$ gives
\begin{align*}
 c_{k+1} &= p s_{k+1}+pt a_k+ a_1 s_k +
   \sum_{j=2}^{k-1} a_j(s_{k+1-j}-a_{k+1-j}) \\
 &= ps_{k+1}+(\beta-2a_1)a_k + a_1s_k + \sum_{j=2}^{k-1} a_jb_{k+1-j} \\
 &= \sum_{j=0}^{k+1} a_jb_{k+1-j}.
\end{align*}
In other words, $a(x)$ and $b(x)$ provide a factorization of $f(x)$ in 
$\Z[[x]]$, as claimed.

%%%%%%%%%%%%%%%%%%%%%%%%%%%%%%%%%%%%%%%%%%%%%%%%%%%%%%%%%%%%%%%%%%%%%

\end{document}